\documentclass[12pt]{article}

\usepackage{amsmath,amssymb,amsthm,amsfonts,diagrams}
\usepackage{nicematrix,tikz}
\usetikzlibrary{fit}
\tikzset{highlight/.style={rectangle,
fill=gray!50,
rounded corners = 0.5 mm, 
inner sep=1pt,
fit=#1}}

\def\Hom{{\rm Hom}}
\def\ind{{\rm ind}}
\def\ad{\mathop{\rm ad}}
\def\im{\mathop{\rm im}}

\def\phi{\varphi}
\def\g{\mathfrak g}

\def\h{\mathfrak h}
\def\F{\mathbb F}

\makeatletter \let\@@pmod\pmod
\DeclareRobustCommand{\pmod}{\@ifstar\@pmods\@@pmod}
\def\@pmods#1{\mkern4mu({\operator@font mod}\mkern 6mu#1)}
\makeatother

\makeatletter
\def\foo#1\endgraf\unskip#2\foo{\def\row@to@buffer{#1\endgraf\unskip\unskip#2}}
\expandafter\foo\row@to@buffer\foo
\makeatother

\newtheorem{theorem}{Theorem}[section]
\newtheorem{lemma}[theorem]{Lemma}

\newtheorem{definition}[theorem]{Definition}
\newtheorem{rem}[theorem]{Remark}

\theoremstyle{remark}

\providecommand{\keywords}[1]{\noindent{Keywords:} #1}
\providecommand{\classify}[1]{\noindent{Mathematics Subject
    Classification:} #1}

\title{Cohomology of Restricted Twisted Heisenberg Lie Algebras}

\author{
Yong Yang\\
College of Mathematics and System Science\\
Xinjiang University\\
Urumqi 830046, China\\
yangyong195888221@163.com}

\date{}

\begin{document}
\maketitle

\begin{abstract}
  Over an algebraically closed field $\F$ of
  characteristic $p>0$, the restricted twisted Heisenberg Lie
  algebras  are studied. We use the 
  Hochschild-Serre spectral sequence relative to its Heisenberg ideal to compute the trivial cohomology. 
 The ordinary 1- and 2-cohomology spaces are used  to compute the restricted 1- and 2-cohomology
  spaces  and describe the
  restricted 1-dimensional central extensions, including explicit
  formulas for the Lie brackets and $-^{[p]}$-operators. 
\end{abstract}

{\footnotesize
\keywords{restricted Lie algebra;
  twisted Heisenberg algebra; (restricted) cohomology; central extension}

\classify{17B50; 17B56}}

\section{Introduction}
The definition of  twisted Heisenberg Lie algebras was motivated by the connection with differential geometry. More specifically, the twisted Heisenberg Lie algebra is 
 the tangent algebra of a twisted Heisenberg group (certain semidirect product of a Heisenberg group with the real numbers) \cite{T}.
 Its importance is due to the relationship between twisted Heisenberg algebras and Lorentzian manifolds, for examples see \cite{A}, \cite{G}. Algebraically, the twisted Heisenberg Lie algebra can be viewed as a certain 1-dimensional extension of the corresponding  Heisenberg Lie algebra as shown in the following definitions:
 \begin{definition}\rm
  For a non-negative integer $m$, the \emph{Heisenberg Lie algebra}
  $\mathfrak{h}_{m}$ is the $(2m+1)$-dimensional vector space over a
  field $\F$ spanned by the elements
  \[\{e_1,\ldots , e_{2m},e_{2m+1}\}\] with the non-vanishing Lie
  brackets
  \[[e_i,e_{m+i}]=e_{2m+1},\] for $1\le i\le m$.
\end{definition} 
 
\begin{definition}\rm
  For a non-negative integer $m$, the \emph{twisted Heisenberg Lie algebra}
  $\mathfrak{h}^{\lambda}_{m}$, parameterized by $\lambda=(\lambda_1,\ldots,\lambda_m)\in(\F^{\times})^{m}$, is the $(2m+2)$-dimensional vector space over a
  field $\F$ spanned by the elements
  \[\{e_1,\ldots , e_{2m},e_{2m+1},e_{2m+2}\}\] with the non-vanishing Lie
  brackets
  \[[e_i,e_{m+i}]=e_{2m+1},\ [e_{2m+2},e_i]=\lambda_ie_{m+i},\ [e_{2m+2},e_{m+i}]=\lambda_ie_i\] for $1\le i\le m$.
\end{definition}

The
cohomology with trivial coefficients of Heisenberg Lie algebras 
was computed by Santharoubane  \cite{S}  over
fields of characteristic zero  in 1983.
  In 2005, Sk\"oldberg  \cite{SK} used algebraic
Morse theory to compute the Poincar\'e polynomial of Heisenberg Lie algebras
 over fields of characteristic $p=2$. In 2008,  Cairns and
 Jambor \cite{CJ} gave the $n$-th Betti number of
$\mathfrak{h}_m$ over  fields of  prime characteristic for $n \le m$.

 In 1937, Jacobson  \cite{J} introduced the notion of restricted Lie algebras over fields of positive
characteristic and proved that  a Lie algebra is restrictable precisely when all $p$-powers of the inner derivations are still inner. 
This gives us many examples of restricted Lie algebras, like 2-step nilpotent Lie algebras,
 the derivation algebra
$\text{Der}(A)$ of any algebra $A$, the Lie algebra of an algebraic
group, the primitive elements of an irreducible co-commutative Hopf
algebra, etc. However certain technical tools are available for
restricted Lie algebras which are not available for  arbitrary modular Lie algebras, see \cite{Se} for examples.

The restricted structures on $\mathfrak{h}_{m}$ were  classified and characterized by Jacobson's theorem in \cite{EFY}. However much less is known about  restricted  twisted Heisenberg algebras.
In this note, we 
classify the family of non-isomorphic restricted structures
on the twisted Heisenberg Lie algebra $\mathfrak{h}_m^{\lambda,\mu}$
parameterized by
the elements $\mu \in \F^{2m+2}$. 
By 
considering $\mathfrak{h}_m^{\lambda,\mu}$ as a restricted one-dimensional
central extension of the restricted Heisenberg Lie algebra,
we compute the ordinary and
restricted cohomology spaces 
$H^q(\mathfrak{h}_m^{\lambda,\mu})$ and
$H_*^q(\mathfrak{h}_m^{\lambda,\mu})$ for $q=1,2$, and  give the bracket structures and
$[p]$-operators for the corresponding restricted one-dimensional
central extensions of $\mathfrak{h}_m^{\lambda,\mu}$.

\section{Restricted Lie algebras and cohomology}

In this section, we recall the definition of a restricted Lie algebra,
the Chevalley-Eilenberg cochain complex,
and the restricted cochain complex \cite{EFu}. Everywhere in this
section, $\F$ denotes an algebraically closed field of characteristic
$p$ and $\g$ denotes a finite dimensional Lie algebra over $\F$ with
an ordered basis $\{e_1,\dots , e_n\}$.  For $j\geq 2$ and $g_1,\ldots,g_ j\in \mathfrak{g}$, we denote the $j$-fold bracket
$$[g_1,g_2,g_3,\ldots,g_j]=[[\ldots[[g_1,g_2],g_3],\ldots],g_j].$$

\subsection{Restricted Lie Algebras}
Suppose that $\mathrm{char}\ \F=p>0$. A \emph{restricted Lie algebra} $\mathfrak{g}$ over $\F$ is a Lie
algebra $\mathfrak{g}$ over $\F$ together with a mapping
$\cdot^{[p]}:\g\to\g$, written $g\mapsto g^{[p]}$, such that for all
$a\in\F$, and all $g,h\in \mathfrak{g}$
\begin{itemize}
\item[(1)] $(a g)^{[p]}=a^{p}g^{[p]}$;
\item[(2)]
  $(g+h)^{[p]}=g^{[p]}+h^{[p]}+ \sum\limits_{i=1}^{p-1} s_i(g,h)$,
  where $is_i(g,h)$ is the coefficient of $t^{i-1}$ in the formal expression
  $(\mathrm{ad}(tg+h))^{p-1}(g)$; and
\item[(3)] $(\mathrm{ad}\ g)^{p}=\mathrm{ad}\ g^{[p]}$.
\end{itemize}

The mapping $\cdot^{[p]}$ is called a \emph{$[p]$-operator} on
$\mathfrak{g}$ and a  Lie algebra is \emph{restrictable} if it is possible to define a
$[p]$-operator on it. By Jacobson's theorem  \cite{J,SF},  a  finite dimensional Lie algebra $\g$ is
restrictable  if and only if it admits a basis $\{e_1,\dots, e_n\}$ such that $(\ad e_i)^p$ is an inner derivation for
all $1\le i\le n$. Choosing $e_i^{[p]}\in\g$ with
$\ad e_i^{[p]}=(\ad e_i)^p$ for all $i$ completely determines a
$[p]$-operator on $\mathfrak{g}$.


\subsection{Chevalley-Eilenberg Lie algebra cohomology}

We only describe the
Chevalley-Eilenberg cochain spaces $C^q(\g)=C^q(\g,\F)$
for $q=0,1,2,3$ and differentials $d^q:C^q(\g)\to C^{q+1}(\g)$
for $q=0,1,2$. For more details on the Chevalley-Eilenberg cochain
complex we refer the reader to \cite{ChE} and \cite{F}. Set
$C^0 (\g)=\F$ and $C^q (\g)= (\wedge^q\g)^*$ for $q=1,2,3$. We will use the
following bases, ordered lexicographically, throughout the paper.
\begin{align*}
  C^0 (\g):& \{1\}&\\
  C^1 (\g):&  \{e^k\ |\ 1\le k\le n\}&\\
  C^2 (\g):& \{e^{i,j}\ |\ 1\le i<j\le n\}&\\
  C^3 (\g):& \{e^{u,v,w}\ |\ 1\le u<v<w\le n\}&\\
\end{align*}
Here $e^k$, $e^{i,j}$ and $e^{u,v,w}$ denote the dual vectors of the
basis vectors $e_k\in \g$,
$e_{i,j}=e_i\wedge e_j\in \wedge^2\g $ and
$e_{u,v,w}=e_u\wedge e_v\wedge e_w\in \wedge^3\g$, respectively.
The differentials $d^q:C^q (\g)\to C^{q+1}(\g)$ are defined for $\psi\in C^1 (\g)$,
$\phi\in C^2 (\g)$ and $g,h,f\in\g$ by\small
\begin{align*}
  d^0: C^0 (\g)\to C^1 (\g),  &\  d^0=0&\\
  d^1:C^1 (\g)\to C^2 (\g), &\   d^1(\psi)(g\wedge h)=\psi([g,h])&\\
  d^2:C^2 (\g)\to C^3 (\g), &\  d^2(\phi)(g\wedge h\wedge f)=\phi([g,h]\wedge f)-\phi([g,f]\wedge h)+\phi([h,f]\wedge g). &\\
\end{align*}\normalsize
The mappings $d^q$ satisfy $d^{q}d^{q-1}=0$ and 
$H^q(\g)=H^q(\g,\F)=\ker(d^q)/\im(d^{q-1})$.

\subsection{Restricted Lie algebra cohomology}

In this subsection, we recall the definitions and results on the  
restricted cochain complex  only for the case of trivial
coefficients.
For a more general
treatment of the cohomology of restricted Lie algebras, we refer the reader to
 \cite{EFu} and \cite{Fe}.

Given $\phi\in C^2(\g)$, a mapping $\omega:\g\to\F$ is {\bf
  $\phi$-compatible} if for all $g,h\in\g$ and all $a\in\F$
\\
  
$\omega(a g)=a^p \omega (g)$ and
\begin{equation}
  \label{starprop}
  \omega(g+h)=\omega(g)+\omega(h) + \sum_{\substack{g_i=\mbox{\rm\scriptsize $g$
        or $h$}\\ g_1=g, g_2=h}}
  \frac{1}{\#(g)}\phi([g_1,g_2,g_3,\dots,g_{p-1}]\wedge g_p)
\end{equation}
where $\#(g)$ is the number of factors $g_i$ equal to $g$. We define
\[\Hom_{\rm Fr}(\g,\F) = \{f:\g\to\F\ |\ f(a x+b
  y)=a^pf(x)+b^pf(y),a,b\in\F; x,y\in \g\}\]  to be
the space of {\it Frobenius homomorphisms} from $\g$ to $\F$. A mapping
$\omega:\g\to\F$ is $0$-compatible if and only if
$\omega\in \Hom_{\rm Fr}(\g,\F)$.

Given $\phi\in C^2(\g)$, we can define a mapping
$\omega: \mathfrak{g} \to \F $ that is $\phi$-compatible by assigning the values of $\omega$ arbitrarily on a
basis for $\mathfrak{g}$ and use (1) to extend it to  $\mathfrak{g}$ and this $\omega$ is  unique because its values are completely
determined by $\omega(e_i)$ and $\phi$ (c.f. \cite{EFi2}). In
particular, given $\phi$, we can define $\tilde\phi(e_i)=0$ for all $i$ and use
(\ref{starprop}) to determine a unique $\phi$-compatible mapping
$\widetilde\phi:\g\to\F$. Note that, in general, $\widetilde\phi\ne 0$
but $\widetilde\phi (0)=0$. Moreover, If $\phi_1,\phi_2\in C^2(\g)$
and $a\in\F$, then
$\widetilde{(a\phi_1+\phi_2)} = a\widetilde\phi_1 + \widetilde\phi_2$.
\\

Given $\zeta\in C^3(\g)$,  a mapping $\eta:\g\times \g\to\F$ is {\bf
  $\zeta$-compatible} if for all $a\in\F$ and all $g,h,h_1,h_2\in\g$,
$\eta(\cdot,h)$ is linear in the first coordinate,
$\eta(g,a h)=a^p\eta(g,h)$ and
\begin{align*}
  \eta(g,h_1+h_2) &=
                    \eta(g,h_1)+\eta(g,h_2)-\nonumber \\
                  & \sum_{\substack{l_1,\dots,l_p=1 {\rm or} 2\\ l_1=1,
  l_2=2}}\frac{1}{\#\{l_i=1\}}\zeta (g\wedge
  [h_{l_1},\cdots,h_{l_{p-1}}]\wedge h_{l_{p}}).
\end{align*}
The restricted cochain spaces are defined as $C^0_*(\g)=C^0 (\g)$,
$C^1_*(\g)=C^1 (\g)$,
\[C^2_*(\g)=\{(\phi,\omega)\ |\ \phi\in C^2 (\g), \omega:\g\to\F\
  \mbox{\rm is $\phi$-compatible}\}\]
\[C^3_*(\g)=\{(\zeta,\eta)\ |\ \zeta\in C^3 (\g),
  \eta:\g\times\g\to\F\ \mbox{\rm is $\zeta$-compatible}\}.\] 

For $1\le i\le n$, define $\overline e^i:\g\to\F$ by
$\overline e^i \left(\sum_{j=1}^n a_j e_j\right ) = a_i^p.$ The set
$\{\overline e^i\ |\ 1\le i\le n\}$ is a basis for
$\Hom_{\rm Fr}(\g,\F)$. Since
\[\dim C^2_*(\g) = \binom{\dim\g+1}{2}=\binom{\dim\g}{2}+\dim\g,\]
we have that
\begin{equation*}
  \{(e^{i,j},\widetilde{e^{i,j}})\ |\ 1\le i<j\le n\} \cup \{(0,\overline e^i)\ |\ 1\le i\le n\}
\end{equation*}
is a basis for $C^2_*(\g)$. We will use this basis in the following computations.

Define $d_*^0=d^0$. For $\psi\in C^1_*(\g)$, define the mapping
$\ind^1(\psi):\g\to\F$ by $\ind^1(\psi)(g)=\psi(g^{[p]})$. The mapping
$\ind^1(\psi)$ is $d^1(\psi)$-compatible for all $\psi\in C^1_*(\g)$,
and the differential $d^1_*:C^1_*(\g)\to C^2_*(\g)$ is defined by
\begin{equation}
  d^1_*(\psi) = (d^1(\psi),\ind^1(\psi)).
\end{equation}
For $(\phi,\omega)\in C^2_*(\g)$, define the mapping
$\ind^2(\phi,\omega):\g \times\g \to\F$ by the formula
\[\ind^2(\phi,\omega)(g,h)=\phi(g\wedge h^{[p]})-\varphi([g,\underbrace{h,\ldots,h}_{p-1}],h).\]
The mapping $\ind^2(\phi,\omega)$ is $d^2(\phi)$-compatible for all
$\phi\in C^2(\g)$, and the differential $d^2_*:C^2_*(\g)\to C^3_*(\g)$
is defined by
\begin{equation}
  d^2_*(\phi,\omega) =
  (d^2(\phi),\ind^2(\phi,\omega)). 
\end{equation}
These mappings $d_*^q$ satisfy $d_*^{q}d_*^{q-1}=0$ and we define
\[H_*^q(\g)=H_*^q(\g,\F)=\ker(d_*^q)/\im(d_*^{q-1}). \]

Note that (with trivial coefficients) if $\omega_1$ and $\omega_2$ are
both $\phi$-compatible, then
$\ind^2(\phi, \omega_1)=\ind^2(\phi, \omega_2)$.

\begin{lemma}\cite{EFY}
  \label{swap}
  If $(\phi,\omega)\in C^2_*(\g)$ and $\phi=d^1(\psi)$ with
  $\psi\in C^1 (\g)$, then $(\phi,\ind^1(\psi))\in C^2_*(\g)$ and
  $\ind^2(\phi,\omega)=\ind^2(\phi,\ind^1(\psi))$.
\end{lemma}

If $\g$ is a  restricted Lie algebra, there is a six-term exact sequence \cite{H,Viv} that
relates the ordinary and restricted 1- and 2-cohomology spaces:
\begin{diagram}[LaTeXeqno]
  \label{sixterm}
  0 &\rTo &H^1_*(\g)&\rTo &H^1(\g)&\rTo&\Hom_{\rm Fr}(\g,\F) & \rTo \\
  & \rTo & H^2_*(\g)&\rTo &H^2(\g)&\rTo^\Delta&\Hom_{\rm
    Fr}(\g,H^1(\g)). &
\end{diagram}
The mapping
\[\Delta: H^2(\g)\to\Hom_{\rm Fr}(\g,H^1(\g))\] in
(\ref{sixterm}) is given explicitly by
\begin{equation}\label{delta}\Delta_\phi(g)\cdot h =\varphi(g\wedge[h,\underbrace{g,\ldots,g}_{p-1}])-\phi(g^{[p]}\wedge
  h)
    \end{equation}
where $g,h\in\g$.

\section{Ordinary cohomology $H^q(\h_m^{\lambda})$ for $q=1,2$}
Suppose $I$ is an ideal of a Lie algebra $\mathfrak{g}$. Then there is a convergent spectral sequence called the Hochschild-Serre spectral sequence \cite{HS} such that
\begin{equation*}
E_{2}^{r,s}=H^{r}(\mathfrak{g}/I,H^{s}(I))\Longrightarrow H^{r+s}(\mathfrak{g}).
\end{equation*}

Here we associate cohomology of a twisted Heisenberg Lie algebra with cohomology of its Heisenberg Lie algebra by means of Hochschild-Serre spectral sequences.
Regarding $\h_{m}$ as an ideal of $\h_{m}^{\lambda}$, we  have
\begin{equation*}
E_{2}^{r,s}=H^{r}(\h_{m}^{\lambda}/\h_{m},H^{s}(\h_{m}))=H^{r}(\F e_{2m+2},H^{s}(\h_{m}))\Longrightarrow H^{r+s}(\h_{m}^{\lambda}).
\end{equation*}
So we have $E_{\infty}^{r,s}=E_{2}^{r,s}=0$ for $r\geq 2$ and $E_{\infty}^{0,s}=E_{2}^{0,s}$, $E_{\infty}^{1,s}=E_{2}^{1,s}$. Then
 we have the following lemma.

\begin{lemma}\label{1.1}
$H^{k}(\h_{m}^{\lambda})=
H^{k}(\h_{m})^{\F e_{2m+2}}\oplus H^{1}(\F e_{2m+2},H^{k-1}(\h_{m}))$ for $k\geq 0$.
\end{lemma}

The cohomology of Heisenberg Lie algebras $\h_m$ with trivial
coefficients was first studied by Santharoubane in \cite{S} (in
characteristic $0$) and by Cairns and Jambor in \cite{CJ} (in
characteristic $p>0$).

\begin{lemma}\label{H}\cite{S,CJ}
  For $q\leq m$,

  (1) if $\mathrm{char}\ \F=0$,
  \[H^{q}(\h_m)=\bigwedge^{q}(\h_m/\F e_{2m+1})^{\ast}/(d
    e^{2m+1}\wedge\bigwedge^{q-2}(\h_m/\F e_{2m+1})^{\ast});\]

  (2) if $\mathrm{char}\ \F=p>0$,
  \begin{align*}
    H^{q}(\h_m)&=\bigwedge^{q}(\h_m/\F
                 e_{2m+1})^{\ast}/(de^{2m+1}\wedge\bigwedge^{q-2}(\h_m/\F
                 e_{2m+1})^{\ast}) \\
               &\oplus((d e^{2m+1})^{p-1}\wedge
                 e^{2m+1}\wedge\bigwedge^{q-2p+1}(\h_m/\F
                 e_{2m+1})^{\ast})
  \end{align*}
  where $d e^{2m+1}=\sum^{m}_{i=1}e^{i,m+i}$.
\end{lemma}

If  $q=1$ or $q=2$, then the second term in part (2) of
Lemma~\ref{H} vanishes so the cohomology spaces do not depend on $p$.
In particular
\begin{equation}\label{o1}
H^1(\h_m)=(\h_m/\F e_{2m+1})^{\ast}
\end{equation}
and
\begin{equation}\label{o2}
H^2(\h_m)=\bigwedge^{2}(\h_m/\F e_{2m+1})^{\ast}/(d
  e_{2m+1}^{\ast}).
\end{equation}
For $1\leq i<j\leq m$,
put
$\delta_{\lambda_i=\pm\lambda_j}=1$ when $\lambda_i=\lambda_j$ or  $\lambda_i=-\lambda_j$ (mod $p$) and $\delta_{\lambda_i=\pm\lambda_j}=0$, otherwise.
By Lemma \ref{1.1} and Eqs. (\ref{o1})-(\ref{o2}), we get the following theorem.

\begin{theorem}
  \label{maintheorem1} 
  For $m\ge 1$ and $\lambda\in(\F^{\times})^{m}$,
  
(1)    $H^{1}(\h_{m}^{\lambda})= \mathbb{F}e^{2m+2}$;

 (2)
 a basis of $H^{2}(\h_{m}^{\lambda})$ consists of  the
  classes of the cocycles
\begin{eqnarray*}
   &&\{\delta_{\lambda_i=\pm\lambda_j}\left(e^{i,j}-\lambda_i\lambda_j^{-1}e^{m+i,m+j}\right), \delta_{\lambda_i=\pm\lambda_j}\left(e^{i,m+j}-\lambda_i\lambda_j^{-1}e^{m+i,j}\right)\\
   &&\mid 1\leq i<j\leq m\} \bigcup\{ e^{i,m+i}\mid 1\leq i\leq m-1\}.
\end{eqnarray*}
 In particular,
$$\dim H^{2}(\h_{m}^{\lambda})=2Card\{(i,j)\mid \lambda_i=\pm \lambda_j,1\leq i<j\leq m\}+m-1.$$
\end{theorem}

\begin{rem}
Another approach for determining the cohomology $H^{1}(\h_{m}^{\lambda})$
is the definition of the first cohomology \cite[(1.4.2)]{F}:
$$H^{1}(\h_{m}^{\lambda})=(\h_{m}^{\lambda}/[\h_{m}^{\lambda},\h_{m}^{\lambda}])^{\ast}=(\h_{m}^{\lambda}/\h_{m})^{\ast}=\F e^{2m+2}.$$
\end{rem}

\section{Restricted structures on the modular Lie algebra
  $\mathfrak{h}^{\lambda}_m$}
In this section, we classify all restricted structures on the twisted Heisenberg Lie
algebra over a field of prime characteristic.

For $e_{2m+2}$,
we have
$$(\ad e_{2m+2})^{p}: \left\{
                             \begin{array}{ll}
                               e_i\mapsto \lambda^{2}_i e_{i} ,  e_{m+i}\mapsto \lambda^{2}_i e_{m+i}, & \hbox{$p=2$;}\\
                              e_i\mapsto \lambda^{p}_i e_{m+i}, e_{m+i}\mapsto \lambda^{p}_i e_i , & \hbox{$p>2$;}                              
                             \end{array}
                                                                                   \right.$$
where $1\leq i\leq m$. So  $(\mathrm{ad}\ e_{2m+2})^{p}$ is an inner derivation if and only if $p>2$ and $\lambda^{p-1}_1=\cdots=\lambda^{p-1}_m$. In this case, we denote by $|\lambda|=\lambda^{p-1}_i$ for $1\leq i\leq m$ and we have $(\mathrm{ad}\ e_{2m+2})^{p}=\mathrm{ad}\ (|\lambda|e_{2m+2})$.  When $p>2$,
we have $(\mathrm{ad}\ e_i)^{p}=0$ for $1\leq i\leq 2m+1$. By Jacobson's theorem \cite{J}, we have the following theorem:
\begin{theorem}
For $m\ge 1$ and $\lambda\in(\F^{\times})^{m}$,
 $\mathfrak{h}^{\lambda}_m$ is restrictable if and only if $p>2$ and $\lambda^{p-1}_1=\cdots=\lambda^{p-1}_m$.
\end{theorem}

From now on, we assume $p > 2$ and $\lambda^{p-1}_1=\cdots=\lambda^{p-1}_m=|\lambda|$.  
We can use Jacobson's theorem \cite{J} and
define a restricted Lie algebra $[p]$-operator on $\mathfrak{h}^{\lambda}_m$ by
choosing $e_i^{[p]}\in \mathfrak{h}^{\lambda}_m$ so that
$\ad e_{2m+2}^{[p]}=(\mathrm{ad}\ e_{2m+2})^{p}=\mathrm{ad}\ (|\lambda|e_{2m+2})$ and $\ad e_i^{[p]}=(\ad e_i)^p=0$  
for all $i$. For each $i$, we
choose a scalar $\mu_i\in \F$ and set

$$e_i^{[p]}=\mu_i e_{2m+1},\quad
e_{2m+2}^{[p]}=|\lambda|e_{2m+2}+\mu_{2m+2} e_{2m+1}.$$
We let
$\mu=(\mu_1,\dots, \mu_{2m+2})$, and we denote the
corresponding restricted Lie algebra $\mathfrak{h}_m^{\lambda,\mu}$.

Since
$[[[\h_m^{\lambda,\mu},\h_m^{\lambda,\mu}],\h_m],\h^{\lambda,\mu}_m]=0$,
by the definition of restricted Lie algebras, we have
$s_i(g,h)=0$ for $i\geq 3$, and
$$s_1(g,h)=[g,\underbrace{h,\ldots,h}_{p-1}],\quad 2s_2(g,h)=[g,\underbrace{h,\ldots,h}_{p-2},g]$$
for all $g\in \mathfrak{h}_{m}$ and 
$h\in \mathfrak{h}_{m}^{\lambda,\mu}$.
 From this we
get that if $g=\sum^{2m+2}_{i=1} a_i e_i\in \mathfrak{h}_m^{\lambda,\mu}$,
then
\begin{align}\label{res}
      \begin{split}
  g^{[p]}&=a_{2m+2}^{p-1}|\lambda|\sum_{i=1}^{2m}a_ie_i+a_{2m+2}^{p}|\lambda|e_{2m+2}\\
  &+\left(\sum^{2m+2}_{i=1}a^{p}_i\mu_i+2^{-1}a_{2m+2}^{p-2}\sum_{i=1}^{m}\lambda_{i}^{p-2}(a_i^{2}-a_{m+i}^{2})\right) e_{2m+1}.
  \end{split}
\end{align}

Denote by $E_{ij}$ the $(2m+1)\times (2m+1)$ matrix with a single nonzero
 entry equal to $1$ in the $(i,j)$-position.
 
\begin{theorem}\label{1}
  Suppose that there are two $p$-operators $\cdot^{[p]}$ and $-^{[p]'}$
  on $\mathfrak{h}^{\lambda}_{m}$ defined by two linear forms
  $\mu,\mu':\h^{\lambda}_m\to\F$, respectively.  Then
  $\mathfrak{h}_{m}^{\lambda,\mu}$ and $\mathfrak{h}_{m}^{\lambda,\mu'}$ are
  isomorphic if and only if 
  there exist an invertible   matrix
  $A=(a_{ij})\in M_{2m+1}(\F)$,
  $(k_1,\dots, k_{2m}, k_{2m+2})\in\F^{2m+1}$ and $k_{2m+1}\in\F^\times$, such that the equations hold for any $i=1,\ldots,m$; $j=1,\ldots,2m,2m+2$ and 
    $(a_1,\ldots, a_{2m+2})\in \F^{2m+2}$, 
    \begin{itemize}

  \item[(1)]  \begin{align*}
  A\left(
    \begin{array}{ccc}
      0 & I_m & 0\\
      -I_m & 0 & 0\\
      0 & 0 & 0\\
    \end{array}
  \right)A^{t} =&k_{2m+1}\left(
    \begin{array}{ccc}
      0 & I_m & 0\\
      -I_m & 0 & 0\\
      0 & 0 & 0\\
    \end{array}
  \right)  \\
  &+ \sum_{i=1}^{m} k_i\lambda_i(E_{2m+1,m+i}-E_{m+i,2m+1})\\
  &+ \sum_{i=1}^{m} k_{m+i}\lambda_i(E_{2m+1,i}-E_{i,2m+1}) ;
    \end{align*}

  \item[(2)]  
    \begin{align*}  
   \lambda_iA(E_{2m+1,m+i}-E_{m+i,2m+1})A^{t} =&\sum_{j=1}^{m}(E_{2m+1,m+j}-E_{m+j,2m+1})\lambda_ja_{j,i}\\
  &+\sum_{j=1}^{m}(E_{2m+1,j}-E_{j,2m+1})\lambda_ja_{m+j,i};
      \end{align*}

  \item[(3)]   
     \begin{align*}  
  \lambda_iA(E_{2m+1,i}-E_{i,2m+1})A^{t} =&\sum_{j=1}^{m}(E_{2m+1,m+j}-E_{m+j,2m+1})\lambda_ja_{j,m+i}\\
  &+\sum_{j=1}^{m}(E_{2m+1,j}-E_{j,2m+1})\lambda_ja_{m+j,m+i};
      \end{align*}

\item[(4)]  
       \begin{align*}  
  \sum_{i=1}^{m}(E_{2m+1,m+i}-E_{m+i,2m+1})\lambda_ia_{i,2m+2}
  +\sum_{i=1}^{m}(E_{2m+1,i}-E_{i,2m+1})\lambda_ia_{m+i,2m+2}=0;
      \end{align*}    
      
  \item[(5)]     
 $a_{2m+2}^{p-1}\mathfrak{a}_j=\mathfrak{a}_{2m+2}^{p-1}\mathfrak{a}_j$;

      \item[(6)]
   \begin{align*}
   \begin{split}
   &a_{2m+2}^{p-1}|\lambda| \sum^{2m}_{i=1}a_ik_i+a_{2m+2}^{p}|\lambda| k_{2m+2} \\
&+k_{2m+1}\left(\sum^{2m+2}_{i=1}a^{p}_i\mu_i+2^{-1}a_{2m+2}^{p-2}\sum_{i=1}^{m}\lambda_{i}^{p-2}(a_i^{2}-a_{m+i}^{2})\right) \\
=&\sum^{2m}_{i=1}\mathfrak{a}^{p}_i\mu'_i+\mathfrak{a}_{2m+2}^{p}\mu'_{2m+2}+\left(\sum_{i=1}^{2m+2}a_ik_i\right)^{p}\mu'_{2m+1}     \\    
&+2^{-1}\mathfrak{a}_{2m+2}^{p-2}\sum_{i=1}^{m}\lambda_{i}^{p-2}(\mathfrak{a}_i^{2}-\mathfrak{a}_{m+i}^{2});
    \end{split}
    \end{align*}
          \end{itemize}
where $\mathfrak{a}_i=\sum_{j=1}^{2m}a_ja_{j,i}+a_{2m+2}a_{2m+2,i}.$
    \end{theorem}

\begin{proof}
  If $\Psi: \mathfrak{h}_m^{\lambda,\mu} \to \mathfrak{h}_m^{\lambda,\mu'}$ is
  a restricted Lie algebra isomorphism, then we can write
  \[\left(
      \begin{array}{c}
        \Psi(e_1) \\
        \vdots \\
        \Psi(e_{2m}) \\
        \Psi(e_{2m+2}) \\
              \end{array}
    \right) =A\left(
      \begin{array}{c}
        e_1 \\
        \vdots \\
        e_{2m} \\
        e_{2m+2}\\
      \end{array}
    \right) +\left(
      \begin{array}{c}
        k_1 \\
        \vdots \\
        k_{2m} \\
        k_{2m+2}\\
      \end{array}
    \right)e_{2m+1}\] and $\Psi(e_{2m+1})=k_{2m+1}e_{2m+1}$ where
  $A\in M_{2m+1}(\F)$ is an invertible matrix,
  $k=(k_1,\ldots,k_{2m},k_{2m+2})\in\F^{2m+1}$ and $k_{2m+1}\in\F^\times$.  If
  $g,h \in \mathfrak{h}_{m}^{\lambda,\mu}$ and we write
  $g=\sum^{2m+2}_{i=1}a_ie_i$, $h=\sum^{2m+2}_{i=1}b_ie_i$,
  then \[\Psi(g)=(a_1,\ldots,a_{2m},a_{2m+2})A\left(
      \begin{array}{c}
        e_1 \\
        \vdots \\
        e_{2m} \\
        e_{2m+2}\\
      \end{array}
    \right) +\left(\sum^{2m+2}_{i=1}a_ik_i\right)e_{2m+1}\]
  and
  \[\Psi(h)=(b_1,\ldots,b_{2m},b_{2m+2})A\left(
      \begin{array}{c}
        e_1 \\
        \vdots \\
        e_{2m} \\
        e_{2m+2}\\
      \end{array}
    \right) +\left(\sum^{2m+2}_{i=1}b_ik_i\right)e_{2m+1}.\] Moreover,
  we have

  \begin{align*}
  [\Psi(g),\Psi(h)]&=(a_1,\ldots,a_{2m},a_{2m+2}) A\left(
      \begin{array}{ccc}
        0 & I_m   & 0  \\
        -I_m & 0   & 0    \\
        0 & 0 & 0
      \end{array}
    \right)A^{t} \left(
      \begin{array}{c}
        b_1 \\
        \vdots \\
        b_{2m} \\
        b_{2m+2}
      \end{array}
    \right) e_{2m+1} \\
     &+ \sum_{i=1}^{m}(a_1,\ldots,a_{2m},a_{2m+2}) \lambda_i A(E_{2m+1,m+i}-E_{m+i,2m+1})A^{t} \left(
      \begin{array}{c}
        b_1 \\
        \vdots \\
        b_{2m} \\
        b_{2m+2}
      \end{array}
    \right)e_i\\
    &+\sum_{i=1}^{m}(a_1,\ldots,a_{2m},a_{2m+2}) \lambda_i A(E_{2m+1,i}-E_{i,2m+1})A^{t} \left(
      \begin{array}{c}
        b_1 \\
        \vdots \\
        b_{2m} \\
        b_{2m+2}
      \end{array}
    \right)e_{m+i}.
  \end{align*}    
   From $\Psi([g,h])=[\Psi(g),\Psi(h)]$, we get Eqs. (1)-(4).
     On the
  other hand, $\Psi$ preserves the restricted $[p]$-structure so that
  $\Psi(g^{[p]})=\Psi(g)^{[p]'}$ for
  $g \in \mathfrak{h}_m^{\lambda,\mu}(p)$.  If we set
  $g=\sum^{2m+2}_{i=1} a_i e_i$ and $(a_1,\ldots,a_{2m},a_{2m+2})A=(\mathfrak{a}_1,\ldots,\mathfrak{a}_{2m},\mathfrak{a}_{2m+2})$, then we have

    \begin{align}\label{2.1.1}
    \begin{split}
       \Psi(g^{[p]})&=a_{2m+2}^{p-1}|\lambda|\sum_{i=1}^{2m}\mathfrak{a}_ie_i+a_{2m+2}^{p-1}|\lambda|\mathfrak{a}_{2m+2}e_{2m+2}\\
    & +\Bigg(a_{2m+2}^{p-1}|\lambda|\sum^{2m}_{i=1}a_ik_i+a_{2m+2}^{p}|\lambda| k_{2m+2}    \\    
&+k_{2m+1}\left(\sum^{2m+2}_{i=1}a^{p}_i\mu_i+2^{-1}a_{2m+2}^{p-2}\sum_{i=1}^{m}\lambda_{i}^{p-2}(a_i^{2}-a_{m+i}^{2})\right)   
    \Bigg)e_{2m+1},
      \end{split}
  \end{align}
and
\begin{align}\label{2.2.2}
 \begin{split}
      \Psi(g)^{[p]'}&=\mathfrak{a}_{2m+2}^{p-1}|\lambda|\sum_{i=1}^{2m}\mathfrak{a}_ie_i+\mathfrak{a}_{2m+2}^{p}|\lambda|e_{2m+2}\\
    & +\Bigg(\sum^{2m}_{i=1}\mathfrak{a}^{p}_i\mu'_i+\mathfrak{a}_{2m+2}^{p}\mu'_{2m+2}+\left(\sum_{i=1}^{2m+2}a_ik_i\right)^{p}\mu'_{2m+1}   \\    
&+2^{-1}\mathfrak{a}_{2m+2}^{p-2}\sum_{i=1}^{m}\lambda_{i}^{p-2}(\mathfrak{a}_i^{2}-\mathfrak{a}_{m+i}^{2})  
    \Bigg)e_{2m+1}.
      \end{split}
\end{align}
  Comparing equations (\ref{2.1.1}) with (\ref{2.2.2}), we get Eqs. (5) and (6).
  Conversely, if there exists an invertible matrix
   $A=(a_{ij})\in M_{2m+1}(\F)$,
  $k=(k_1,\dots, k_{2m+2})\in\F^{2m+2}$ and $k_{2m+1}\in\F^\times$   which satisfy the above conditions,
  it is easy to check that the argument above is reversible and we
  obtain an isomorphism between the restricted Lie algebras
  $\mathfrak{h}_m^{\lambda,\mu}$ and $\mathfrak{h}_m^{\lambda,\mu'}$.
\end{proof}

\begin{rem}
If $g=\sum^{2m+1}_{i=1} a_i e_i\in \mathfrak{h}_m$, (\ref{res}) gives the restricted structures on the Heisenberg Lie algebra $\mathfrak{h}_m$ \cite{EFY}:
\[g^{[p]}=\left(\sum^{2m+1}_{i=1}a^{p}_i\mu_i\right) e_{2m+1}.\]
Moreover,  $\mathfrak{h}_m$ is a restricted ideal of $   \mathfrak{h}^{\lambda,\mu}_m$.
\end{rem}

\section{Restricted cohomology $H^q_*(\h_m^{\lambda,\mu})$ for $q=1,2$}

By definition, $[\h_m^{\lambda,\mu},\h_m^{\lambda,\mu}]= \h_{m}$, and the
$[p]$-operator formula (\ref{res})  imply
\[ [\h_m^{\lambda,\mu},\h_m^{\lambda,\mu}]+\langle (\h_m^{\lambda,\mu})^{[p]}\rangle_\F=
\h_m^{\lambda,\mu}.\]
 For any restricted Lie algebra $\g$,
 \cite[Theorem 2.1]{H} states that
\[
  H^1_*(\g)=(\g/([\g,\g]+\langle \g^{[p]}\rangle_\F))^*.\] It
follows that $H^1_*(\h_m^{\lambda,\mu})=0$, and the
six-term exact sequence (\ref{sixterm}) decouples to the exact
sequence

\begin{diagram}
  0 &\rTo &\F e^{2m+2}&\rTo &\Hom_{\rm Fr}(\h_m^{\lambda,\mu},\F)&\rTo& H^2_*(\h_m^{\lambda,\mu})& \rTo \\
  &    \rTo & H^2(\h_m^{\lambda,\mu}) &\rTo^\Delta&\Hom_{\rm Fr}(\h_m^{\lambda,\mu},\F e^{2m+2}). 
\end{diagram}
If we let $g=\sum a_ie_i\in\h_m^{\lambda,\mu}$ and $\omega:\h_m^{\lambda,\mu}\to \F$ be
any $\phi$-compatible map to a 2-cocycle $\phi$, then  Theorem \ref{maintheorem1} (2)  and (\ref{delta}) give
\begin{eqnarray*}
  \Delta_\phi(g)\cdot e_{2m+2} 
                         &=&\varphi(g\wedge [e_{2m+2},\underbrace{g,\ldots,g}_{p-1}])-\phi(g^{[p]}\wedge
                             e_{2m+2})\\
                        &=&-a_{2m+2}^{p-2}|\lambda|\varphi(g,g)\\
                         &=&0.                                                                               
\end{eqnarray*}
Then  $\Delta=0$, and we have a splitting exact sequence
\begin{diagram}
  0&\rTo& \Hom_{\rm Fr}(\h_m^{\lambda,\mu},\F)/\F \overline e^{2m+2}& \rTo & H^2_*(\h_m^{\lambda,\mu}
  )
   &\rTo & H^2(\h_m^{\lambda,\mu} ) &\rTo&0,
\end{diagram}
where the mapping $ \Hom_{\rm Fr}(\h_m^{\lambda,\mu},\F)/\F \overline e^{2m+2} \to  H^2_*(\h_m^{\lambda,\mu}) $ sends $\omega$ to the
class of $ (0,\omega)$ and the image is a $(2m+1)$-dimensional
subspace of $ H^2_*(\h_m^{\lambda,\mu})$ spanned by the classes $(0,\overline
e^i)$ for $1\leq i\leq 2m+1$. It follows that
 \[H^2_*(\h_m^{\lambda,\mu})\cong\Hom_{\rm Fr}(\h_m^{\lambda,\mu},\F)/\F \overline e^{2m+2}\oplus
    H^2(\h_m^{\lambda,\mu}).\] 
Theorem \ref{maintheorem1} (2) implies the following theorem:

\begin{theorem}
  \label{maintheorem}
  For $m\ge 1$, $\lambda\in(\F^{\times})^{m}$ and any form $\mu:\h^{\lambda,\mu}_m\to\F$,
$$\dim H^2_*(\h_m^{\lambda,\mu})=2Card\{(i,j)\mid \lambda_i=\pm \lambda_j,1\leq i<j\leq m\}+3m$$
  and a basis consists of the
  classes of the cocycles
  \[ \{ \delta_{\lambda_i=\pm\lambda_j}(e^{i,j}-\lambda_i\lambda^{-1}_je^{m+i,m+j},
  \widetilde{e^{i,j}}-\lambda_i\lambda^{-1}_j\widetilde{e^{m+i,m+j}}),
    \]
  \[  \delta_{\lambda_i=\pm\lambda_j} (e^{i,m+j}-\lambda_i \lambda^{-1}_je^{m+i,j}, \widetilde{ e^{i,m+j}}-\lambda_i\lambda^{-1}_j\widetilde{e^{m+i,j}})   
  \mid1\leq i<j\leq m\} \]
  \[\bigcup \{(e^{i,m+i},\widetilde{e^{i,m+i}})\ |\ 1\le i\le m-1\}\bigcup \{(0,\overline e^i)\ |\ 1\le i\le 2m+1\}.\]
\end{theorem}

\section{Restricted one-dimensional central extensions of
  $\h_m^{\lambda,\mu}$}
Similar to the case of  ordinary Lie
algebras
[\cite{F}, Chapter 1, Section 4.6], restricted one-dimensional
central extensions of a restricted Lie algebra $\g$
are parameterized by the restricted cohomology group $H^2_*(\g)$
[\cite{H}, Theorem 3.3]. If $(\phi,\omega) \in C^2_{*}(\g)$, then the corresponding restricted
one-dimensional central extension $\mathfrak{G}=\g\oplus\F c$ has Lie bracket and
$[p]$-operator defined by
\begin{align}\label{genonedimext}
  \begin{split}
  [g,h]_\mathfrak{G}&=[g,h]_{\g} + \phi(g\wedge h)c\\
  [g,c]_\mathfrak{G}&=0\\
  g^{[p]_\mathfrak{G}}&=g^{[p]_{\g}} + \omega(g)c\\
  c^{[p]_\mathfrak{G}}&=0
  \end{split}
\end{align}
where $[\cdot,\cdot]_{\g}$ and $\cdot^{{[p]}_{\g}}$ denote the Lie
bracket and $[p]$-operator in $\g$, respectively (\cite{EFu},
equations (26) and (27)). We can the equations (\ref{genonedimext}) together with
Theorem~\ref{maintheorem} to explicitly describe the restricted
one-dimensional central extensions of $\h_m^{\lambda,\mu}$. Let $g=\sum a_ie_i$ and $h=\sum b_ie_i$ denote two arbitrary elements of
$\h_m^{\lambda,\mu}$. For $1\leq s<t\leq 2m$, the equation (\ref{starprop}) gives

\begin{align}\label{p}
\begin{split}
  \widetilde{e^{s,t}}
  (g)&=  \widetilde{e^{s,t}}
  \left(\sum_{i=1}^{2m+1}a_ie_i+a_{2m+2}e_{2m+2}\right)\\
  &=2^{-1}a_{2m+2}^{p-2}\sum_{i,j=1}^{2m}a_ia_je^{s,t}([e_i,\underbrace{e_{2m+2},\ldots,e_{2m+2}}_{p-2}],e_j)\\
        &= -2^{-1}a_{2m+2}^{p-2}
       \bigg(
       \sum_{i=1}^{m}\sum_{j=1}^{2m}\lambda_{i}^{p-2}a_i a_j
       e^{s,t}(e_{m+i},e_j) \\
       &+
       \sum_{i=1}^{m}\sum_{j=1}^{2m}\lambda_{i}^{p-2}a_{m+i}a_je^{s,t}(e_{i},e_j) \bigg).
 \end{split}
\end{align}

If $1\le i<j\le m$ such that $\lambda_i=\lambda_j$ or $\lambda_i=-\lambda_j$, and $\mathfrak{H}_{ij}=\h_m^{\lambda.\mu}\oplus \F c$ denotes
the one-dimensional restricted central extension of $\h_m^{\lambda,\mu}$
determined by the cohomology class of the restricted cocycle
$$ \left(e^{i,j}-\lambda_i\lambda^{-1}_je^{m+i,m+j},
  \widetilde{e^{i,j}}-\lambda_i\lambda^{-1}_j\widetilde{e^{m+i,m+j}}\right),$$
   then (\ref{genonedimext}) and (\ref{p}) give the bracket and
$[p]$-operator in $\mathfrak{H}_{ij}$:
\begin{align*}
  \begin{split}
    [g,h]_{ \mathfrak{H}_{ij}}& =[g,h]_{\h_m^{\lambda,\mu}}+(a_ib_j-a_jb_i-\lambda_i\lambda^{-1}_ja_{m+i}b_{m+j}+\lambda_i\lambda^{-1}_ja_{m+j}b_{m+i})c;\\
    g^{[p]_{\mathfrak{H}_{ij}}} & = g^{[p]_{\h_m^{\lambda,\mu}}}-2^{-1}a_{2m+2}^{p-2}(\lambda^{p-2}_{i} a_{m+i}a_{j}-\lambda_{j}^{p-2}a_{m+j}a_{i}-|\lambda|\lambda_j^{-1}a_i a_{m+j}\\
    &+\lambda_i\lambda_j^{p-3}a_ja_{m+i}) c.
  \end{split}
\end{align*}

If $1\le i<j\le m$ such that $\lambda_i=\lambda_j$ or $\lambda_i=-\lambda_j$,  and $\mathfrak{H}_{i,m+j}=\h_m^{\lambda.\mu}\oplus \F c$ denotes
the one-dimensional restricted central extension of $\h_m^{\lambda,\mu}$
determined by the cohomology class of the restricted cocycle
$$(e^{i,m+j}-\lambda_i \lambda^{-1}_je^{m+i,j}, \widetilde{ e^{i,m+j}}-\lambda_i\lambda^{-1}_j\widetilde{e^{m+i,j}}),$$
then (\ref{genonedimext}) and (\ref{p}) give the bracket and
$[p]$-operator in $\mathfrak{H}_{i,m+j}$:
\begin{align*}
  \begin{split}
    [g,h]_{\mathfrak{H}_{i,m+j}} & =[g,h]_{\h_m^{\lambda,\mu}}+(a_ib_{m+j}-a_{m+j}b_i-\lambda_i\lambda^{-1}_ja_{m+i}b_{j}+\lambda_i\lambda^{-1}_ja_{j}b_{m+i})c;\\
      g^{[p]_{\mathfrak{H}_{i,m+j}}} & = g^{[p]_{\h_m^{\lambda,\mu}}}-2^{-1}a_{2m+2}^{p-2}(\lambda^{p-2}_{i} a_{m+i}a_{m+j}-\lambda_{j}^{p-2}a_{j}a_{i}-|\lambda|\lambda_j^{-1}a_i a_{j}\\
    &+\lambda_i\lambda_j^{p-3}a_{m+j}a_{m+i}) c.
  \end{split}
\end{align*}

If $1\le i\le m-1$ and $\mathfrak{H}_{i,m+i}=\h_m^{\lambda.\mu}\oplus \F c$ denotes
the one-dimensional restricted central extension of $\h_m^{\lambda,\mu}$
determined by the cohomology class of the restricted cocycle
$(e^{i,m+i},\widetilde{e^{i,m+i}})$, 
then (\ref{genonedimext}) and (\ref{p}) give the bracket and
$[p]$-operator in $\mathfrak{H}_{i,m+i}$:
\begin{align*}
  \begin{split}
    [g,h]_{\mathfrak{H}_{i,m+i}} & =[g,h]_{\h_m^{\lambda,\mu}}+(a_ib_{m+i}-a_{m+i}b_i)c;\\
      g^{[p]_{\mathfrak{H}_{i,m+i}}} & = g^{[p]_{\h_m^{\lambda,\mu}}}-2^{-1}a_{2m+2}^{p-2}\lambda^{p-2}_{i} (a^{2}_{m+i}-a_{i}^{2}) c.
  \end{split}
\end{align*}

If $1\le i\le 2m+1$ and $\mathfrak{H}_i=\h_m^{\lambda.\mu}\oplus \F c$ denotes
the one-dimensional restricted central extension of $\h_m^{\lambda,\mu}$
determined by the cohomology class of the restricted cocycle
$(0,\overline e^i)$, then (\ref{genonedimext}) gives the bracket and
$[p]$-operator in $\mathfrak{H}_i$:
\begin{align*}
  \begin{split}
    [g,h]_{\mathfrak{H}_i} & =[g,h]_{\h_m^{\lambda,\mu}};\\
    g^{[p]_{\mathfrak{H}_i}} & = g^{[p]_{\h_m^{\lambda,\mu}}}+ a_i^p c.
  \end{split}
\end{align*}
The central extensions $\mathfrak{H}_i$ form a basis for the $(2m+1)$-dimensional
space of restricted one-dimensional central extensions that split as
ordinary Lie algebra extensions (c.f. \cite{EFi2}).

\end{document}